\documentclass[12pt]{amsart}

\usepackage{tikz-cd}
\usepackage{amsfonts}
\usepackage{amsthm}
\usepackage{xfrac}
\usepackage{faktor}
\usepackage{amssymb}
\usepackage{enumitem}
\usepackage{microtype}
\usepackage{amsthm}
\usepackage{amsmath}
\usepackage{amssymb}
\usepackage{mathtools}
\usepackage{tikz-cd}
\usepackage{microtype}
\usepackage{calrsfs}
\DeclareMathAlphabet{\pazocal}{OMS}{zplm}{m}{n}
\usepackage{mathabx}
\usepackage{amsmath}
\usepackage[all]{xy}
\usepackage{undertilde}
\usepackage[numbers]{natbib}

\title{Finiteness criteria for Gorenstein flat dimension and stability}
\author{Ilias Kaperonis and Dimitra-Dionysia Stergiopoulou}

\thanks{Research supported by the Hellenic Foundation for Research and Innovation (H.F.R.I.) under the ``1st Call
for H.F.R.I. Research Projects to support Faculty members and Researchers and the procurement of high-cost
research equipment grant”, project number 4226.}

\newtheorem{Lemma}{Lemma}[section]
\newtheorem{Proposition}[Lemma]{Proposition}
\newtheorem{Theorem}[Lemma]{Theorem}
\newtheorem{Corollary}[Lemma]{Corollary}
\newtheorem{Remark}[Lemma]{Remark}
\newtheorem{Definition}[Lemma]{Definition}

\begin{document}

\begin{abstract} Projectively coresolved Gorenstein flat modules were introduced recently by Saroch and Stovicek and were shown to be Gorenstein projective. While the relation between Gorenstein projective and Gorenstein flat modules is not well understood, the class of projectively coresolved Gorenstein flat modules is contained in the class of Gorenstein flat modules. This paper proves necessary and sufficient conditions for a module of finite Gorenstein flat dimension to be projectively coresolved Gorenstein flat, or of finite flat dimension. Stability results for the class of projectively coresolved Gorenstein flat modules are also established. 

\end{abstract}

\maketitle
%\tableofcontents

\section{Introduction}
Auslander and Bridger \cite{AB} introduced the concept of $G$-dimension for commutative Noetherian rings. This was extended to modules over any ring $R$ through the notion of a Gorenstein projective module by Enochs and Jenda \cite{EJ,EJ2}, who also defined the notion of Gorenstein injective and Gorenstein flat modules. The relative homological dimensions based on these modules were defined in \cite{H1}, which is the standard reference for these notions. Later, Saroch and Stovicek \cite{S-S} introduced the notion of projectively coresolved Gorenstein flat modules (PGF modules, for short). Over a ring $R$, these modules are the syzygies of the acyclic complexes of projective modules that remain acyclic after applying the functor $I\otimes_R  \_\!\_$ for every injective module $I$, and hence they are Gorenstein flat. The class of PGF modules is contained in the class of Gorenstein projective modules, as shown in \cite[Theorem 4.4]{S-S}.

Maximal Cohen-Macaulay approximations were introduced and studied by Auslander and Buchweitz in \cite{AB2} for finitely generated modules over a commutative Gorenstein local ring. These approximations have been generalized in \cite[Lemma 2.17]{Ch-Fr-H} and \cite[Theorem 2.10]{H1}, where it is shown that for every ring $R$ and every $R$-module $M$ of finite Gorenstein projective dimension there exist approximation sequences of $R$-modules of the form
\begin{equation}\label{eq1}0 \rightarrow K \rightarrow G \rightarrow M \rightarrow  0,
\end{equation}
\begin{equation}\label{eq2}0 \rightarrow M \rightarrow A\rightarrow G' \rightarrow  0,
\end{equation}
where $G,G'$ are Gorenstein projective and $K,A$ have finite projective dimension. Emmanouil and Talelli \cite{Em-Ta} showed that exact sequences as above satisfy a weak form of uniqueness, as in the classical Schanuel’s lemma.

Analogous approximation sequences have been contructed in \cite[Theorem 4.11]{S-S} concerning Gorenstein flat, PGF and flat modules. More explicitly, for every ring $R$ and every Gorenstein flat $R$-module $N$ there exist short exact sequences of $R$-modules of the form 
\begin{equation}\label{eq3}
	0 \rightarrow F \rightarrow L \rightarrow N \rightarrow 0 ,
\end{equation}
\begin{equation}\label{eq4}
	0 \rightarrow N \rightarrow F' \rightarrow L' \rightarrow 0 ,
\end{equation}
where $L,L'$ are PGF and $F,F'$ are flat. 

Emmanouil \cite[Theorem 2.1]{Em} studied the class of modules of finite Gorenstein flat dimension and obtained approximation sequences of these modules by PGF modules and modules of finite flat dimension. More precisely, he generalized the sequences (\ref{eq2}), (\ref{eq3}) by constructing for every ring $R$ and every $R$-module $M$ of finite Gorenstein flat dimension short exact sequences of $R$-modules of the form
\begin{equation}\label{eq5}
 0 \rightarrow F \rightarrow L \rightarrow M \rightarrow 0 ,
\end{equation}
\begin{equation}\label{eq6}
 0 \rightarrow M \rightarrow F' \rightarrow L' \rightarrow 0 ,
\end{equation}
where $L,L'$ are PGF and $F,F'$ have finite flat dimension. We will show that an exact sequence of the form (\ref{eq5}) or (\ref{eq6}) satisfies also
a weak form of uniqueness, as in the classical Schanuel’s lemma.

In this paper we prove necessary and sufficient conditions for a module of finite Gorenstein flat dimension to be PGF, or of finite flat dimension. Furthermore, we establish stability results for the class of projectively coresolved Gorenstein flat modules. The content and results of this paper may be described in some detail as follows.

In Section 2, we establish notation, terminology and preliminary results that will be used in the sequel.
Using the exact sequence (\ref{eq5}), in Section 3 we will describe the precise sense in which $L$ may be regarded as an approximation of $M$ by a PGF module. In particular, we consider the stabilization $\mathfrak{FF}\textrm{-}R\textrm{-Mod}$ of the category of $R$-modules with respect to the class of modules of finite flat dimension and denote by $\mathfrak{FF}\mbox{-}\overline{{\tt GFlat}}(R)$ (respectively, by $\mathfrak{FF}\mbox{-}{\tt PGF}(R)$) the full subcategory of $\mathfrak{FF}\textrm{-}R\textrm{-Mod}$ consisting of the modules of finite Gorenstein flat dimension (respectively, of the PGF modules). As described in Theorem \ref{thr1} (ii), we obtain an additive functor $ \mu ' : \mathfrak{FF}\mbox{-}\overline{{\tt GFlat}}(R) \rightarrow
\mathfrak{FF}\mbox{-}{\tt PGF}(R),$ which is right adjoint to the inclusion functor $\mathfrak{FF}\mbox{-}{{\tt PGF}(R)} \hookrightarrow
\mathfrak{FF}\mbox{-}\overline{{\tt GFlat}}(R)$. Moreover, the vanishing of the counit of this adjunction gives a characterization of the modules of finite flat dimension among modules of finite Gorenstein flat dimension. 

We also consider a complete hereditary cotorsion pair $(\mathcal{K},\mathcal{L})$ such that ${\mathcal{L}}$ contains all projective $R$-modules and $\mathcal{L}$ is closed under kernels of epimorphisms. Moreover, we consider the stabilization $\mathfrak{L}\textrm{-}R\textrm{-Mod}$ of the category of $R$-modules with respect to the class $\mathcal{L}$ and denote by $\mathfrak{L}\mbox{-}\mathcal{K}$ the full subcategory of $\mathfrak{L}\textrm{-}R\textrm{-Mod}$ consisting of the modules in $\mathcal{K}$. Then, we prove that there exists an additive functor $ \mu : \mathfrak{L}\mbox{-}R \textrm{-Mod} \rightarrow \mathfrak{L}\mbox{-}\mathcal{K},$ which is right adjoint to the inclusion functor $\mathfrak{L}\mbox{-}\mathcal{K} \hookrightarrow
\mathfrak{L}\mbox{-}R \textrm{-Mod}$ (see Theorem \ref{thr1} (i)) and obtain similar results. Moreover, we obtain equivalent conditions of a Gorenstein flat module to be flat. 

Using the exact sequence (\ref{eq6}), in Section 4 we consider the $R$-module $F'$ as an approximation of $M$ by a module of finite flat dimension. In particular, we consider the stabilization $\mathfrak{PGF}\textrm{-}R\textrm{-Mod}$ of the category of $R$-modules with respect to the class of PGF modules and denote by $\mathfrak{PGF}\mbox{-}\overline{{\tt GFlat}}(R)$ (respectively, by $\mathfrak{PGF}\mbox{-}\overline{{\tt Flat}}(R)$) the full subcategory of $\mathfrak{PGF}\textrm{-}R\textrm{-Mod}$ consisting of the modules of finite Gorenstein flat dimension (respectively, of finite flat dimension). As described in Theorem \ref{theo2} (i), we construct an additive functor $\nu : \mathfrak{PGF}\mbox{-}\overline{{\tt GFlat}}(R)\rightarrow
\mathfrak{PGF}\mbox{-}\overline{{\tt Flat}}(R)$ which is left adjoint to the inclusion functor $\mathfrak{PGF}\mbox{-}\overline{{\tt Flat}}(R) \hookrightarrow \mathfrak{PGF}\mbox{-}\overline{{\tt GFlat}}(R)$. The vanishing of the unit of this adjunction gives a characterization of the PGF modules among modules of finite Gorenstein flat dimension. 

Sather-Wagstaff, Sharif and White \cite{SSW} investigated the modules that arise from the iteration of the very procedure that leads to the Gorenstein projective modules and proved the stability of the classes of Gorenstein projective and Gorenstein injective modules. Bouchiba and Khaloui \cite{BK} proved the analogous stability of the class of Gorenstein flat modules. In Section 5, we prove that the class of PGF $R$-modules is stable under this very Gorenstein process. In particular, we prove that for every exact sequence of PGF $R$-modules $$\textbf{G}=  \cdots \rightarrow G_1 \rightarrow G_0 \rightarrow G^0 \rightarrow G^1 \rightarrow \cdots$$ such that $M\cong \textrm{Im}(G_0 \rightarrow G^0)$ and $I \otimes_R \_\!\_$ preserves exactness of $\textbf{G}$ for every injective $R$-module $I$, the $R$-module $M$ is PGF (see Theorem \ref{finale}). A central role in the proof is played by the subcategory consisting of the $R$-modules $M$ for which there exists a short exact sequence of the form $0\rightarrow M \rightarrow G \rightarrow M \rightarrow 0$, where $G$ is a PGF $R$-module such that $I \otimes_R \_\!\_$ preserves exactness of this sequence for every right injective $R$-module $I$. An application of this stability result may be found in \cite[Proposition 2.15]{St}.

\section{Preliminaries}

In this section we establish notation, terminology and preliminary results that will be used in the sequel.

\vspace{0.1in}
 
\noindent{\em Terminology.} 
Unless otherwise specified, all modules considered in this paper will be left modules over an arbitrary ring $R$. We denote by $R$-$\textrm{Mod}$ the category of $R$-modules and by ${\tt Proj}(R)$, ${\tt Flat}(R)$, ${\tt Inj}(R)$ the categories of projective, flat and injective $R$-modules respectively. 

\vspace{0.1in}
%${\tt GProj}(R)$, ${\tt PGF}(R)$, ${\tt GFlat}(R)$ 
\noindent{\em Gorenstein modules.}
An acyclic complex $\textbf{P}$ of projective modules is said to be a complete 
projective resolution if the complex of abelian groups $\mbox{Hom}_R(\textbf{P},Q)$
is acyclic for every projective module $Q$. Then, a module is Gorenstein 
projective if it is a syzygy of a complete projective resolution. We denote by ${\tt GProj}(R)$ the class of Gorenstein projective $R$-modules. The 
Gorenstein projective dimension $\mbox{Gpd}_RM$ of a module $M$ is the 
length of a shortest resolution of $M$ by Gorenstein projective modules. 
If no such resolution of finite length exists, then we write 
$\mbox{Gpd}_RM = \infty$. If $M$ is a module of finite projective dimension, then $M$ has finite Gorenstein projective dimension as well and $\mbox{Gpd}_RM = \mbox{pd}_RM$.

An acyclic complex $\textbf{F}$ of flat modules is said to be a complete flat
resolution if the complex of abelian groups $I \otimes_R \textbf{F}$ is acyclic 
for every injective right module $I$. Then, a module is Gorenstein 
flat if it is a syzygy of a complete flat resolution. We denote by ${\tt GFlat}(R)$ the class of Gorenstein flat $R$-modules. The Gorenstein flat dimension 
$\mbox{Gfd}_RM$ of a module $M$ is the length of a shortest resolution 
of $M$ by Gorenstein flat modules. If no such resolution of finite length
exists, then we write $\mbox{Gfd}_RM = \infty$. If $M$ is a module of 
finite flat dimension, then $M$ has finite Gorenstein flat dimension as
well and $\mbox{Gfd}_RM = \mbox{fd}_RM$ (see Corollary \ref{corolaki}).

The notion of a projectively coresolved Gorenstein flat module (PGF-module, for short) is introduced by Saroch and Stovicek \cite{S-S}. A PGF module is a syzygy of an acyclic complex of projective modules $\textbf{P}$, which is such that the complex of abelian groups $I \otimes_R \textbf{P}$ is acyclic for every injective module $I$. We denote by ${\tt PGF}(R)$ the class of PGF $R$-modules. It is clear that the class ${\tt PGF}(R)$ is contained in ${\tt GFlat}(R)$. Moreover, we have the inclusion ${\tt PGF}(R) \subseteq {\tt GProj}(R)$ (see \cite[Theorem 4.4]{S-S}). Finally, the class of PGF $R$-modules, is closed under extensions, direct sums, direct summands and kernels of epimorphisms. 

%The PGF dimension $\mbox{PGF-dim}_RM$ of a module $M$ is the length of a shortest resolution of $M$ by PGF modules. If no such resolution of finite length exists, then we write $\mbox{PGF-dim}_RM = \infty$. If $M$ is a module of finite projective dimension, then $M$ has finite PGF dimension as well and $\mbox{PGF-dim}_RM = \mbox{pd}_RM$ (see \cite[Proposition 2.2, Corollary 3.7(i)]{DE}).
\medskip
\noindent{\em Cotorsion pairs.} Let $\mathcal{L}$ be a class of $R$-modules. We define its left orthogonal class as $^\perp\mathcal{L}=\{X\in R\text{-}\textrm{Mod}\mid \textrm{Ext}^1(X,L)=0\text{, }  \forall L\in\mathcal{L}\}$. The right orthogonal class of $\mathcal{L}$ is defined dually. An ordered pair of classes, $(\mathcal{K},\mathcal{L})$ is called a cotorsion pair if $\mathcal{K}^\perp=\mathcal{L}$ and $K=$ $^\perp\mathcal{L}$.
A cotorsion pair $(\mathcal{K},\mathcal{L})$ is called complete if for every $R$-module $M$ there are approximation sequences of the form
\[ 0 \rightarrow L \rightarrow K \rightarrow M \rightarrow 0 \,\,\, \,\textrm{and}\,\,\,\, 0 \rightarrow M\rightarrow L' \rightarrow K' \rightarrow 0,\]
 where $K,K'\in\mathcal{K}$ and $L,L'\in\mathcal{L}$. Finally, the cotorsion pair $(\mathcal{K},\mathcal{L})$ is called hereditary if $ \textrm{Ext}^{i>0}(K,L)=0$ for every $L\in\mathcal{L}$ and every $K\in\mathcal{K}$. In that case the class $\mathcal{K}$ is closed under kernels of epimorphisms, while the class $\mathcal{L}$ is closed under cokernels of monomorphisms.

\vspace{0.1in}

\medskip
\noindent The following proposition gives a characterization of PGF modules.
    
\begin{Proposition}\label{pgf}
	The following conditions are equivalent for an $R$-module $M$:
	\begin{enumerate}
		\item[(i)] M is PGF.
		\item[(ii)] M satisfies the following two conditions:
		\begin{itemize}
			\item[(a)] There exists an exact sequence of $R$-modules of the form $0\rightarrow M \rightarrow P^0 \rightarrow P^1 \rightarrow \cdots$, where each $P^i$ is projective, such that $I \otimes_R \_\!\_$ preserves exactness of this sequence for every injective $R$-module $I$.
			\item[(b)] $\textrm{Tor}_i^R(I,M)=0$ for every $i>0$ and every injective $R$-module $I$.
		\end{itemize}
		\item[(iii)] There exists a short exact sequence of $R$-modules of the form $0\rightarrow M \rightarrow P \rightarrow G \rightarrow 0$, where $P$ is projective and $G$ is PGF.
	\end{enumerate}
\end{Proposition}

\begin{proof}This follows immediately using standard arguments.\end{proof}	 
	% Using the definition of  PGF modules, the equivalence $(i)\Leftrightarrow (ii)$ is obtained by standard arguments. Also, by definition, the implication $(i)\Rightarrow (iii)$ is clear.$(iii)\Rightarrow (ii)$: Let $0\rightarrow M \rightarrow P \rightarrow G \rightarrow 0$ be an exact sequence, where $P$ is projective and $G$ is a PGF module. Since $G$ is PGF, the implication $(i)\Rightarrow (ii)$ yields $\textrm{Tor}_i^R(I,G)=0$ for every $i>0$ and every injective module $I$. Let $I$ be an injective module. Then, the short exact sequence $0\rightarrow M \rightarrow P \rightarrow G \rightarrow 0$ induces a long exact sequence of the form $\cdots \rightarrow \textrm{Tor}_{i+1}^R(I,G) \rightarrow \textrm{Tor}_{i}^R(I,M)\rightarrow \textrm{Tor}_{i}^R(I,P)\rightarrow \cdots ,$ where $i>0$, which implies that $\textrm{Tor}_{i}^R(I,M)=0$ for every $i>0$. Moreover, there exists an exact sequence of the form $0\rightarrow G \rightarrow P^0 \rightarrow  P^1 \rightarrow \cdots $, such that $I\otimes_R \_\!\_$ preserves the exactness of this sequence. We obtain an exact sequence of $R$-modules of the form $0 \rightarrow M\rightarrow P\rightarrow P^0 \rightarrow  P^1 \rightarrow \cdots ,$ such that $I\otimes_R \_\!\_$ preserves exactness of this sequence for every injective $R$-module $I$.

\noindent{\em Relative PGF and Gorenstein flat modules.} According \cite{G-I} and \cite{E-I-P} one may consider relative Gorenstein classes. More precisely, if $\mathcal{B}$ is a class of right $R$-modules which contains all injective modules, we call a module $M$ Gorenstein $\mathcal{B}$-flat if it is a syzygy in an acyclic complex of flat modules that remains exact after tensoring by any module in $\mathcal{B}$. Similarly, we call a module $M$ projectively coresolved Gorenstein $\mathcal{B}$-flat if it is a syzygy in an acyclic complex of projective modules that remains exact after tensoring by any module in $\mathcal{B}$. We note that \cite[Remark 38]{G-I} and \cite[Theorem 2.14]{E-I-P} give interesting generalizations of \cite[Theorem 4.11]{S-S}.

\begin{Proposition}\label{1.1}{\rm(\cite[Remark 38]{G-I})} Let ${\tt PGF}_{\mathcal{B}}(R)$ be the class of all projectively coresolved Gorenstein $\mathcal{B}$-flat modules. We also suppose that $({\tt PGF}_{\mathcal{B}}(R),{\tt PGF}_{\mathcal{B}}^\perp (R))$ is a complete hereditary cotorsion pair. Then the following conditions are equivalent for a module $M$.
\begin{enumerate}[label=(\roman*)]
    \item $M$ is Gorenstein $\mathcal{B}$-flat,
    \item There is a short exact sequence of the form 
    \[ 0 \rightarrow F \rightarrow P
     \rightarrow M \rightarrow 0,\]
      where $F\in{\tt Flat}(R)$ and $P\in {\tt PGF}_{\mathcal{B}}(R)$, which remains exact after applying the functor $Hom_R(-,C)$ for every cotorsion\footnote{Cotorsion modules are defined to be the modules in ${\tt Flat}(R)^\perp$} module $C$,
     \item $Ext^1(M,C)=0$ for every cotorsion module $C$ that belongs to ${\tt PGF}_{\mathcal{B}}^\perp (R)$,
     \item There is a short exact sequence of the form 
    \[ 0 \rightarrow M \rightarrow   F'
     \rightarrow P' \rightarrow 0\]
      where $F'\in{\tt Flat}(R)$ and $P'\in {\tt PGF}_{\mathcal{B}}(R)$.
\end{enumerate}

\end{Proposition} 
Note that by \cite[Corollary 4.5 (3)]{S-S} the class of projectively coresolved Gorenstein $\overline{{}_RR^{\text{c}}}$-flat modules is exactly the class of PGF modules, where $\overline{{}_RR^{\text{c}}}$ is the definable closure of the
character module $\textrm{Hom}_{\mathbb{Z}}({}_RR,\mathbb{Q}/\mathbb{Z})$ (for a brief introduction to the theory of definable classes see \cite[Section 2]{S-S}, \cite{Prest}). Then, Proposition \ref{1.1} and \cite[Theorem 4.11]{S-S} imply that the class of Gorenstein $\overline{{}_RR^{\text{c}}}$-flat modules is exactly the class of Gorenstein flat modules.

%{\em Modules of finite Gorenstein $\mathcal{B}$-flat dimension.} Let $\mathbf{GF_{\mathcal{B}}}$ be the class of all Gorenstein $\mathcal{B}$-flat modules.
%In the case that $(\mathbf{GF_{\mathcal{B}}},\mathbf{GF_{\mathcal{B}}}^\perp)$ is a complete hereditary cotorsion pair we define the Gorenstein $\mathcal{B}$-flat dimension of a module $M$ in the standard way, using resolutions by Gorenstein $\mathcal{B}$-flat modules.

%\begin{Lemma} Assume that $(\mathbf{GF_{\mathcal{B}}},\mathbf{GF_{\mathcal{B}}}^\perp)$ is a complete hereditary cotorsion pair. Let $M$ be a module and $n$ an non-negative integer. If
%\[ 0\longrightarrow K \longrightarrow F_{n-1} \longrightarrow
%   \cdots \longrightarrow F_0\longrightarrow M\longrightarrow 0 , \]
%\[ 0\longrightarrow K' \longrightarrow F'_{n-1} \longrightarrow
 %  \cdots \longrightarrow F'_0\longrightarrow M\longrightarrow 0 , \]
  % are two exact sequences with $F_{i},F'_{i}\in\mathbf{GF_{\mathcal{B}}}$ for all $0\leq i\leq n-1$, then $K\in\mathbf{GF_{\mathcal{B}}}$ if and only if $K'\in\mathbf{GF_{\mathcal{B}}}$.
%\end{Lemma}

%\begin{proof} From the assumption that $(\mathbf{GF_{\mathcal{B}}},\mathbf{GF_{\mathcal{B}}}^\perp)$ is a complete hereditary cotorsion pair we also get that the class of Gorenstein $\mathcal{B}$-flat modulesis projectively resolving, closed under direct sums and direct summands. The result follows be \cite[Lemma 3.12]{A-B}.

%\end{proof}

In Sections 3, 4 and 6 we will be concerned with PGF modules and modules of finite Gorenstein flat dimension but our results transfer to the relative case under the assumption that $({\tt PGF}_{\mathcal{B}}(R),{\tt PGF}_{\mathcal{B}}^\perp (R))$ is a complete hereditary cotorsion pair.

\vspace{0.15in}

Even though the exact sequences (1),(2),(3),(4),(5),(6) are not unique, they do satisfy a weak form of uniqueness, as shown below. 

%Please note that by \cite[Theorem 4.4]{S-S} we get ${\tt PGF}(R)\subseteq {}^\perp{\tt Flat}(R)$ and by the fact that ${\tt PGF}^\perp(R)$ is closed under cokernels of monomorphisms (cf. \cite[Theorem 4.9]{S-S}) we get that all modules of finite flat dimension are in ${\tt PGF}^\perp(R)$ .

\begin{Lemma}[Schanuel]\label{Scanuel}

Let $M$ be an $R$-module.

(i) Assume that
$0 \rightarrow F \rightarrow L\rightarrow M \rightarrow 0$
and
$0 \rightarrow \overline{F} \rightarrow \overline{L}
   \rightarrow M \rightarrow 0$
are two short exact sequences of $R$-modules with $\mbox{Ext}_{R}^{1} \! \left( L,\overline{F} \right)=0$ and $\mbox{Ext}_{R}^{1} \! \left( \overline{L},F \right)=0$. Then,
there is an isomorphism $L \oplus \overline{F} \simeq \overline{L} \oplus F$.

(ii) Assume that
$0 \rightarrow M \rightarrow F' \rightarrow L' \rightarrow 0$
and
$0 \rightarrow M \rightarrow \overline{F'}
   \rightarrow \overline{L'} \rightarrow 0$
are two short exact sequences of $R$-modules with $\mbox{Ext}_{R}^{1} \! \left(L',\overline{F'} \right)=0$ and $\mbox{Ext}_{R}^{1} \! \left(\overline{L'},F' \right)=0$.
Then, there is an isomorphism
$L' \oplus \overline{F'} \simeq \overline{L'} \oplus F'$.
\end{Lemma}

\begin{proof}(i) Since $\textrm{Ext}_{R}^{1}( L,\overline{F})=0$, there exists an $R$-linear map $f : L \rightarrow \overline{L}$
	making the right square in the following diagram with exact rows commutative:
	\[
	\begin{array}{ccccccccc}
		0 & \rightarrow & F & \rightarrow & L & \rightarrow & M &
		\rightarrow & 0 \\
		& & \!\!\! {\scriptstyle{g}} \downarrow & &
		\!\!\! {\scriptstyle{f}} \downarrow
		& & \parallel & & \\
		0 & \rightarrow & \overline{F} & \rightarrow & \overline{L} &
		\rightarrow & M & \rightarrow & 0
	\end{array}
	\]
	Then, there exists a unique $R$-linear map $g : F \rightarrow \overline{F}$
	which makes the left square of the diagram commutative as well. We may view the
	pair $(g,f)$ as a quasi-isomorphism between the complexes
	$0 \rightarrow F \rightarrow L \rightarrow 0$ and
	$0 \rightarrow \overline{F} \rightarrow \overline{L} \rightarrow 0$
	and consider the corresponding mapping cone, in order to obtain a short exact sequence of the form
	\[ 0 \rightarrow F \rightarrow L \oplus \overline{F}
	\rightarrow \overline{L} \rightarrow 0 . \]
	Since $\textrm{Ext}_{R}^{1} (\overline{L},F)=0$, the above exact sequence splits. Consequently, there exists an
	isomorphism $L \oplus \overline{F} \simeq \overline{L} \oplus F$, as needed.
	
	(ii) Since $\textrm{Ext}_{R}^{1} (L',\overline{F'})=0$, there exists an $R$-linear map $g' : F' \rightarrow \overline{F'}$
	making the left square in the following diagram with exact rows commutative:
	\[
	\begin{array}{ccccccccc}
		0 & \rightarrow & M & \rightarrow & F' & \rightarrow & L' &
		\rightarrow & 0 \\
		& & \parallel & & \!\!\! {\scriptstyle{g'}} \downarrow
		& & \!\!\! {\scriptstyle{f'}} \downarrow & & \\
		0 & \rightarrow & M & \rightarrow & \overline{F'} &
		\rightarrow & \overline{L'} & \rightarrow & 0
	\end{array}
	\]
	Then, there exists a unique $R$-linear map $f' : L' \rightarrow \overline{L'}$
	which makes the right square of the diagram commutative as well. We may view the
	pair $(g',f')$ as a quasi-isomorphism between the complexes
	$0 \rightarrow F' \rightarrow L' \rightarrow 0$ and
	$0 \rightarrow \overline{F'} \rightarrow \overline{L'}
	\rightarrow 0$
	and consider the corresponding mapping cone, in order to obtain a short exact sequence of the form
	\[ 0 \rightarrow F' \rightarrow L' \oplus \overline{F'}
	\rightarrow \overline{L'} \rightarrow 0 . \]
	Since $\textrm{Ext}_{R}^{1}( \overline{L'},F')=0$, the above exact sequence splits. Consequently, there exists an
	isomorphism $L' \oplus \overline{F'} \simeq \overline{L'} \oplus F'$, as needed.
\end{proof}

\begin{Lemma}\label{lemataki} Let $M$ be a Gorenstein flat $R$-module of finite flat dimension. Then $M$ is flat.
\end{Lemma}

\begin{proof}We let $\textrm{fd}_{R}M=n<\infty$ and assume that $n>0$. Since the $R$-module $M$ is Gorenstein flat, we have $\textrm{Tor}_i^R(J,M)=0$ for every $i>0$ and every injective right $R$-module $J$ (see \cite[Lemma 2.4]{Bennis}). Since $\textrm{fd}_{R}M=n$, there exists a right $R$-module $N$ such that $\textrm{Tor}_n^R(N,M)\neq 0$. Consider an exact sequence of $R$-modules of the form $0\rightarrow N \rightarrow I \rightarrow N'\rightarrow 0$ where the $R$-module $I$ is injective. The short exact sequence above induces a long exact sequence of the form $$\cdots \rightarrow \textrm{Tor}_{n+1}^R(N',M)\rightarrow \textrm{Tor}_{n}^R(N,M)\rightarrow \textrm{Tor}_{n}^R(I,M)\rightarrow\textrm{Tor}_{n}^R(N',M)\rightarrow \cdots.$$ Since $\textrm{Tor}_{n+1}^R(N',M)=0$ and $\textrm{Tor}_n^R(N,M)\neq 0$, we obtain that $\textrm{Tor}_n^R(I,M)\neq 0$ which is a contradiction. We conclude that $n=0$ and hence the $R$-module $M$ is flat.
\end{proof}

\begin{Corollary}\label{corolaki}Let $M$ be an $R$-module of finite flat dimension. Then, $\textrm{Gfd}_R M =\textrm{fd}_R M$.\end{Corollary}

\begin{proof}It suffices to prove the inequality $\textrm{fd}_R M\leq \textrm{Gfd}_R M$, where $\textrm{Gfd}_R M=n<\infty$. Let $\textbf{F}= \cdots \rightarrow F_1 \rightarrow F_0 \rightarrow M \rightarrow 0$ be a flat resolution of $M$. Since every ring $R$ is ${\tt GF}$-closed (see \cite[Corollary 4.12]{S-S}), the $n$-th syzygy module $M_n=\textrm{Im}(F_n \rightarrow F_{n-1})$ is Gorenstein flat by \cite[Theorem 2.8]{Bennis} and $\textrm{fd}_R M_n <\infty$. Consequently, the $R$-module $M_n$ is flat by Lemma \ref{lemataki} and the exact sequence $$0\rightarrow M_n \rightarrow F_{n-1}\rightarrow \cdots \rightarrow F_1 \rightarrow F_0 \rightarrow M \rightarrow 0$$ implies that $\textrm{fd}_R M\leq n$.\end{proof}

\section{Stabilizing with respect to modules of finite flat dimension}

Let $M,N$ be two modules. Since the direct sum of two modules of finite flat dimension has finite flat dimension as well, the set consisting of those $R$-linear maps
$f : M \rightarrow N$ that factor through a module of finite flat dimension is a subgroup of the abelian group $\mbox{Hom}_{R}(M,N)$. We denote by
$\mathfrak{FF}\mbox{-}\mbox{Hom}_{R}(M,N)$ the corresponding quotient
group and let $[f] = [f]_{\mathfrak{FF}}$ be the class of any $R$-linear map
$f \in \mbox{Hom}_{R}(M,N)$ therein. We note that for any three modules $M,N$ 
and $L$ the composition of $R$-linear maps induces a well-defined biadditive map
\[ \mathfrak{FF}\mbox{-}\mbox{Hom}_{R}(N,L) \times
   \mathfrak{FF}\mbox{-}\mbox{Hom}_{R}(M,N)
   \rightarrow \mathfrak{FF}\mbox{-}\mbox{Hom}_{R}(M,L) .\]
We also denote by $\mathfrak{FF}\mbox{-}R\mbox{-Mod}$ the category whose
objects are all modules and whose morphism sets are given by the abelian
groups $\mathfrak{FF}\mbox{-}\mbox{Hom}_{R}(M,N)$ (with composition of
morphisms induced by the composition of $R$-linear maps). Then, Lemma \ref{Scanuel}
implies that the PGF module $G$, which is
defined for any module $M$ of finite Gorenstein flat dimension by (5), is uniquely determined up to isomorphism
as an object of the category $\mathfrak{FF}\mbox{-}R\mbox{-Mod}$.

Let $\mathcal{L}$ be any class of modules closed under finite direct sums. Then, the set consisting of those $R$-linear maps
$f : M \rightarrow N$ that factor through a module in the class $\mathcal{L}$ is a subgroup of the abelian group $\mbox{Hom}_{R}(M,N)$. We denote by $\mathfrak{L}\mbox{-}\mbox{Hom}_{R}(M,N)$ the corresponding quotient
group and by $[f]_{\mathfrak{L}}$ the class of any $f \in \mbox{Hom}_{R}(M,N)$. Moreover, for every class of $R$-modules $\mathcal{K}$, we denote by $\mathfrak{L}\mbox{-}\mathcal{K}$ the category whose objects are all modules in $\mathcal{K}$ and whose morphism sets are given by the abelian groups $\mathfrak{L}\mbox{-}\mbox{Hom}_{R}(M,N)$.

\medskip
%The next lemma is a generalization of \cite[Lemma 2.1]{Em-Ta}.
\begin{Lemma}\label{2.1}
Let $f : M \rightarrow N$ be an $R$-linear map, where $M,N$ are $R\mbox{-modules}$ and consider two classes of $R$-modules $\mathcal{L},\mathcal{K}$ such that $\mathcal{K}\subseteq {}^\perp\mathcal{L}$, ${\tt Proj}(R)\subseteq{\mathcal{L}}$ and $\mathcal{L}$ is closed under finite direct sums and kernels of epimorphisms. We also consider two exact sequences
of $R$-modules
\[ 0 \rightarrow L \stackrel{\iota}{\rightarrow} K
     \stackrel{p}{\rightarrow} M \rightarrow 0
   \;\;\; \mbox{and} \;\;\;
   0 \rightarrow L' \stackrel{\jmath}{\rightarrow} K'
     \stackrel{q}{\rightarrow} N \rightarrow 0 , \]
where $K,K'\in\mathcal{K}$ and $L,L'\in\mathcal{L}$.
Then,
\begin{itemize}
\item[(i)] There exists an $R$-linear map $g : K \rightarrow K'$, such that
$qg = fp$.

\item[(ii)] If $g,g' : K \rightarrow K'$ are two $R$-linear maps with
$qg = fp$ and $qg' = fp$, then
$[g] = [g'] \in \mathfrak{L}\mbox{-}\mbox{Hom}_{R}(K,K')$.

\item[(iii)] If $[f] = [0] \in \mathfrak{L}\mbox{-}\mbox{Hom}_{R}(M,N)$
and $g : K \rightarrow K'$ is an $R$-linear map such that $qg = fp$,
then $[g] = [0] \in \mathfrak{L}\mbox{-}\mbox{Hom}_{R}(K,K')$.
\end{itemize}
\end{Lemma}

\begin{proof}(i) Since $\mbox{Ext}_{R}^{1}(K,L')=0$ the additive
map $q_{*} : \mbox{Hom}_{R}(K,K') \rightarrow \mbox{Hom}_{R}(K,N)$
is surjective. Consequently, there exists an $R$-linear map
$g : K\rightarrow K'$ such that $fp = q_{*}(g)$, as needed.

(ii) Let $g,g' : K \rightarrow K'$ be two $R$-linear maps with
$qg = fp$ and $qg' = fp$.
\[
\begin{array}{ccccccccc}
 0 & \rightarrow & L & \stackrel{\iota}{\rightarrow} & K
   & \stackrel{p}{\rightarrow} & M & \rightarrow & 0 \\
 & & & & {\scriptstyle{g}} \downarrow \downarrow {\scriptstyle{g'}}
 & & \!\!\! {\scriptstyle{f}} \downarrow & & \\
 0 & \rightarrow & L' & \stackrel{\jmath}{\rightarrow} & K'
   & \stackrel{q}{\rightarrow} & N & \rightarrow & 0
\end{array}
\]
Then, $q(g'-g) = qg' - qg = fp - fp = 0$ and hence there exists an
$R$-linear map $h : K \rightarrow L'$, such that $g'-g = \jmath h$.
Since $L'\in\mathcal{L}$ we obtain that
$[g] = [g'] \in \mathfrak{L}\mbox{-}\mbox{Hom}_{R}(K,K')$.

(iii) Without loss of generality we may assume that ${\tt Proj}(R)\subseteq{\mathcal{K}}$. Let $f$ be a linear map which factors as the composition of two $R$-linear maps
$M \stackrel{a}{\rightarrow} \Lambda \stackrel{b}{\rightarrow} N$,
where the module $\Lambda$ belongs to $\mathcal{L}$. Since the class $\mathcal{L}$ is closed under kernels of epimorphisms and every projective module is in $\mathcal{L}$, we construct an exact sequence of $R$-modules of the form
\[ 0 \rightarrow \Lambda' \rightarrow P
     \stackrel{\pi}{\rightarrow} \Lambda \rightarrow 0 \]
where $P$ is projective and
$\Lambda' \in\mathcal{L}$. Using (i) above, we obtain $R$-linear maps
$\alpha : K \rightarrow P$ and $\beta : P \rightarrow K'$, such
that $\pi \alpha = ap$ and $q \beta = b \pi$.
\[
\begin{array}{ccccccccc}
 0 & \rightarrow & L & \stackrel{\iota}{\rightarrow} & K
   & \stackrel{p}{\rightarrow} & M & \rightarrow & 0 \\
 & & & & \!\!\! {\scriptstyle{\alpha}} \downarrow & &
         \!\!\! {\scriptstyle{a}} \downarrow & & \\
 0 & \rightarrow & \Lambda' & \rightarrow & P &
     \stackrel{\pi}{\rightarrow} & \Lambda & \rightarrow & 0 \\
 & & & & \!\!\! {\scriptstyle{\beta}} \downarrow & &
         \!\!\! {\scriptstyle{b}} \downarrow & & \\
 0 & \rightarrow & L' & \stackrel{\jmath}{\rightarrow} & K'
   & \stackrel{q}{\rightarrow} & N & \rightarrow & 0
\end{array}
\]
Consequently, $q(\beta \alpha) = (ba)p = fp$ and hence for every
$R$-linear map $g : K \rightarrow K'$ such that $qg = fp$ we obtain that
$[g] = [\beta \alpha] \in \mathfrak{L}\mbox{-}\mbox{Hom}_{R}(K,K')$
(see (ii) above). Since $[\beta \alpha] = [0] \in \mathfrak{L}\mbox{-}\mbox{Hom}_{R}(K,K')$, we conclude that $[g] = [0] \in \mathfrak{L}\mbox{-}\mbox{Hom}_{R}(K,K')$.\end{proof}

\begin{Remark}\label{remarkaki}Let $\overline{{\tt Flat}}(R)$ be the class of $R$-modules with finite flat dimension and consider a PGF $R$-module $K$. Since the functors $\textrm{Ext}^i_R(K,\_\!\_)$ vanish on the class ${\tt Flat}(R)$ for every $i>0$ (see \cite[Theorem 4.11]{S-S}), a simple inductive argument implies that the functors $\textrm{Ext}^i_R(K,\_\!\_)$ vanish also on the class $\overline{{\tt Flat}}(R)$, and hence ${\tt PGF}(R)\subseteq {}^\perp\overline{{\tt Flat}}(R)$. Moreover, ${\tt Proj}(R)\subseteq\overline{{\tt Flat}}(R)$ and the class $\overline{{\tt Flat}}(R)$ is closed under finite direct sums and kernels of epimorphisms.
\end{Remark}

\begin{Lemma}\label{2.2}
Let $f : M \rightarrow N$ be an $R$-linear map, where $M,N$ are modules of finite Gorenstein flat dimension. We also consider two exact sequences
of $R$-modules
\[ 0 \rightarrow F \stackrel{\iota}{\rightarrow} L
     \stackrel{p}{\rightarrow} M \rightarrow 0
   \;\;\; \mbox{and} \;\;\;
   0 \rightarrow F' \stackrel{\jmath}{\rightarrow} L'
     \stackrel{q}{\rightarrow} N \rightarrow 0 , \]
where $F,F'$ are modules of finite flat dimension and $L,L'$ are PGF modules.
Then,
\begin{itemize}
\item[(i)] There exists an $R$-linear map $g : L \rightarrow L'$, such that
$qg = fp$.

\item[(ii)] If $g,g' : L \rightarrow L'$ are two $R$-linear maps with
$qg = fp$ and $qg' = fp$, then
$[g] = [g'] \in \mathfrak{FF}\mbox{-}\mbox{Hom}_{R}(L,L')$.

\item[(iii)] If $[f] = [0] \in \mathfrak{FF}\mbox{-}\mbox{Hom}_{R}(M,N)$
and $g : L \rightarrow L'$ is an $R$-linear map such that $qg = fp$,
then $[g] = [0] \in \mathfrak{FF}\mbox{-}\mbox{Hom}_{R}(L,L')$.
\end{itemize}
\end{Lemma}

\begin{proof}This is a direct consequence of Lemma \ref{2.1} and Remark \ref{remarkaki}. \end{proof}
	%We can just use the previous lemma since the modules of finite flat dimension are closed under kernels of epimorphisms, every projective module is flat and so has flat dimension zero. Additionally, by \cite[Theorem 4.11]{S-S} we know that ${\tt Flat}(R)\subseteq{\tt PGF}(R)^\perp$, since ${\tt PGF}(R)^\perp$ is closed under cokernels of monomorphisms we can easily deduce that the modules of finite flat dimension are right orthogonal to all PGF modules.

The next propositions shows that under the right assumptions, the categories $\mathfrak{L}\mbox{-}\mathcal{K}$ we defined in the start of this section are nothing more than the well known stable categories.

\begin{Proposition}\label{2.3} Let $\mathcal{L},\mathcal{K}$ be two classes of $R$-modules such that $\mathcal{K}\subset {}^\perp\mathcal{L}$, ${\tt Proj}(R)\subset{\mathcal{L}}$ and $\mathcal{L}$ is closed under kernels of epimorphisms. Consider the category  $\mathfrak{FF}\mbox{-}{{\tt PGF}(R)}$ whose objects are all PGF modules and whose morphism sets are the abelian groups $\mathfrak{FF}\mbox{-}\mbox{Hom}_{R}(M,N)$. Then:
\begin{itemize}
    \item[(i)]  $\mathfrak{L}\mbox{-}\mathcal{K}$ is exactly the stable category of $\mathcal{K}$,
    \item[(ii)] $\mathfrak{FF}\mbox{-}{{\tt PGF}(R)}$ is exactly the stable category of PGF modules,
    \item[(iii)] $\mathfrak{FF}\mbox{-}{{\tt PGF}(R)}$ is  exactly $\mathfrak{PGF}^\perp\mbox{-}{{\tt PGF}(R)}$, where $\mathfrak{PGF}^\perp\mbox{-}{{\tt PGF}(R)}$ is the stabilization of PGF modules with respect to ${{\tt PGF}(R)}^\perp$.
\end{itemize}

\end{Proposition}

\begin{proof}(i) It suffices to prove that if $f:K\rightarrow K'$ is an $R$-linear map which factors through a module in $\mathcal{L}$, then $f$ factors through a projective module for every $K,K'\in\mathcal{K}$. It suffices to show that for every $L\in\mathcal{L}$ every $R$-linear map $g: K\rightarrow L$ factors through a projective module. As in the proof of Lemma \ref{2.1}(ii), we construct an exact sequence of $R$-modules of the form
 $0 \rightarrow L'\rightarrow P \rightarrow L\rightarrow 0,$ where $P$ is projective and $L'\in\mathcal{L}$. Since $\textrm{Ext}^1_R(K,L')=0$, we obtain the exact sequence $0\rightarrow \textrm{Hom}_R(K,L') \rightarrow \textrm{Hom}_R(K,P) \rightarrow \textrm{Hom}_R(K,L) \rightarrow 0$, as needed.
    
    (ii) This follows from (i) and Remark \ref{remarkaki}.
    
    (iii) By \cite[Theorem 4.9]{S-S} the pair $({\tt PGF}(R),{\tt PGF}(R)^\perp)$ is a complete cotorsion pair with ${\tt PGF}(R)^\perp$ thick (especially, it is closed under kernels of epimorphisms). We also have that every PGF module is Gorenstein projective and ${\tt Proj}(R)\subset{{{\tt PGF}(R)}^\perp}$ by \cite[Proposition 2.3]{H1}. The result now follows from (i) and (ii).
\end{proof}

 Let $(\mathcal{K},\mathcal{L})$ be a complete cotorsion pair such that ${\tt Proj}(R)\subset{\mathcal{L}}$ and the class $\mathcal{L}$ is closed under kernels of epimorphisms. We observe that Lemmas \ref{Scanuel}, \ref{2.1} and Corollary \ref{2.2} imply that there are well-defined additive functors
\[ \mu : \mathfrak{L}\mbox{-}\mbox{R-Mod} \rightarrow
           \mathfrak{L}\mbox{-}{\mathcal{K}}  \]
and
\[ \mu ' : \mathfrak{FF}\mbox{-}\overline{{\tt GFlat}}(R) \rightarrow
           \mathfrak{FF}\mbox{-}{\tt PGF}(R) , \]
where $\mu$  maps a $R\mbox{-module}$ $M$ to a module in $\mathcal{K}$ that appears in an approximation sequence of $M$ and $\mu '$ maps a module $M$ of finite Gorenstein flat dimension to the PGF module which appears in the exact sequence (5) of the Introduction. Moreover, $\mu$ maps the class $[f]\in \mathfrak{L}\mbox{-}\mbox{Hom}_R(M,N)$ of an $R$-linear map $f: M \rightarrow N$ to the class $[g]\in \mathfrak{L}\mbox{-}\mbox{Hom}_{R}(K,K')$ of any $R$-linear map $g: K \rightarrow K'$ which satisfies the condition described in Lemma \ref{2.1}(i), while $\mu '$ maps the class $[f]\in \mathfrak{FF}\mbox{-}\mbox{Hom}_R(M,N)$ of an $R$-linear map $f: M \rightarrow N$ to the class $[g]\in \mathfrak{FF}\mbox{-}\mbox{Hom}_{R}(L,L')$ of any $R$-linear map $g: L \rightarrow L'$ which satisfies the condition described in Lemma \ref{2.2}(i). 

The main result of this section is stated below.

\begin{Theorem}\label{thr1}

(i) The additive functor
$ \mu : \mathfrak{L}\mbox{-}\mbox{R-Mod} \rightarrow
           \mathfrak{L}\mbox{-}{\mathcal{K}}$
 is right adjoint to the inclusion functor
$\mathfrak{L}\mbox{-}{\mathcal{K}} \hookrightarrow
 \mathfrak{L}\mbox{-}\mbox{R-Mod}$.
 
 (ii) The additive functor
$ \mu ' : \mathfrak{FF}\mbox{-}\overline{{\tt GFlat}}(R) \rightarrow
           \mathfrak{FF}\mbox{-}{\tt PGF}(R)$
 is right adjoint to the inclusion functor
$\mathfrak{FF}\mbox{-}{{\tt PGF}(R)} \hookrightarrow
 \mathfrak{FF}\mbox{-}\overline{{\tt GFlat}}(R)$.
\end{Theorem}
\begin{proof} We prove only (i) as the proof of (ii) is essentially the same.
Let $N$ be any module and
\[ 0 \rightarrow L \stackrel{\jmath}{\rightarrow} K
     \stackrel{q}{\rightarrow} N\rightarrow 0 , \]
 be an approximation sequence of $N$, where $K\in\mathcal{K}$ and $L\in\mathcal{L}$. We note that the additive map
\[
 [q]_{*} : \mathfrak{L}\mbox{-}\mbox{Hom}_{R}(\Lambda,K) \rightarrow
           \mathfrak{L}\mbox{-}\mbox{Hom}_{R}(\Lambda,N) \]
is natural in both $\Lambda\in\mathcal{K}$ (this is straightforward) and $N$ (this follows
from Lemma \ref{2.1}). We shall establish the adjunction in the statement of
the theorem, by proving that the additive map $[q]_{*}$ above is bijective.

The additive map
\[ q_{*} : \mbox{Hom}_{R}(\Lambda,K) \rightarrow \mbox{Hom}_{R}(\Lambda,N) \]
is surjective because $\mbox{Ext}_{R}^{1} \! \left(\Lambda,L \right)=0$ , whence we get the surjectivity of $[q]_{*}$. As far as the injectivity
of $[q]_{*}$ is concerned, assume that $g : \Lambda \rightarrow K$ is an $R$-linear
map, such that
$[qg] = [q][g] = [q]_{*}[g] = [0] \in \mathfrak{L}\mbox{-}\mbox{Hom}_{R}(\Lambda,N)$.
Then, we may consider the commutative diagram
\[
\begin{array}{ccccccccc}
 0 & \rightarrow & 0 & \rightarrow & \Lambda & = & \Lambda & \rightarrow
   & 0 \\
 & & & & \!\!\! {\scriptstyle{g}} \downarrow & &
         \!\!\!\!\! {\scriptstyle{qg}} \downarrow
 & & \\
 0 & \rightarrow & L & \stackrel{\jmath}{\rightarrow} & K
   & \stackrel{q}{\rightarrow} & N & \rightarrow & 0
\end{array}
\]
and invoke Lemma \ref{2.1}(iii), in order to conclude that
$[g] = [0] \in \mathfrak{L}\mbox{-}\mbox{Hom}_{R}(\Lambda,K)$.\end{proof}
\begin{Corollary}\label{cor27}
	Let $(\mathcal{K},\mathcal{L})$ be a complete hereditary cotorsion pair such that ${\tt Proj}(R)\subset{\mathcal{L}}$ and the class $\mathcal{L}$ is closed under kernels of epimorphisms.
	Then the following conditions are equivalent for an $R$-module $M$.
	\begin{itemize}
	\item[(i)] $M\in\mathcal{L}$,
	
	\item[(ii)] $\mathfrak{L}\mbox{-}\mbox{Hom}_{R}(K,M) = 0$ for all $K\in\mathcal{K}$ and
	
	\item[(iii)] there exists a short exact sequence of $R$-modules
	$0 \rightarrow L \rightarrow K \stackrel{q}{\rightarrow}
	M \rightarrow 0$,
	where $L\in\mathcal{L}$ and $K\in\mathcal{K}$, such
	that $[q] = [0] \in \mathfrak{L}\mbox{-}\mbox{Hom}_{R}(K,M)$.
	\end{itemize}
\end{Corollary}
%\begin{proof}The proof is similar with the proof of Corollary \ref{cor26}.\end{proof}

\begin{proof}
	The implications (i)$\Rightarrow$(ii) and (ii)$\Rightarrow$(iii) are clear
	and hence it suffices to show that (iii)$\Rightarrow$(i). To that end, we
	note that the adjunction isomorphism established by Theorem \ref{thr1}(i)
	implies, in particular, that the additive map
	\[ [q]_{*} : \mathfrak{L}\mbox{-}\mbox{Hom}_{R}(K,K) \rightarrow
	\mathfrak{L}\mbox{-}\mbox{Hom}_{R}(K,M) \]
	is bijective. Since
	$[q]_{*}[1_{K}] = [q][1_{K}] = [q1_{K}] = [q] = [0] \in
	\mathfrak{L}\mbox{-}\mbox{Hom}_{R}(K,M)$,
	we conclude that $[1_{K}] = [0] \in \mathfrak{L}\mbox{-}\mbox{Hom}_{R}(K,K)$;
	hence, the identity map of $K$ factors through a module $L'\in \mathcal{L}$. Since $\mathcal{L}$ is closed under direct summands and $K$ is a direct summand of $L'$ we obtain that $K$ belongs to $\mathcal{L}$ as well. Moreover, the class $\mathcal{L}$ is closed under cokernels of monomorphisms, and hence the short exact sequence $0 \rightarrow L \rightarrow K \rightarrow M \rightarrow 0$ yields $M\in \mathcal{L}$.
\end{proof}

\begin{Corollary}\label{cor26}The following conditions are equivalent for an $R$-module $M$ of finite Gorenstein flat dimension.
\begin{itemize}
\item[(i)] $M$ is an $R$-module of finite flat dimension,

\item[(ii)] $\mathfrak{FF}\mbox{-}\mbox{Hom}_{R}(L,M) = 0$ for every PGF $R$-module $L$ and

\item[(iii)] there exists a short exact sequence of $R$-modules
      $0 \rightarrow F \rightarrow L \stackrel{q}{\rightarrow}
       M \rightarrow 0$,
      where $F$ is a module of finite flat dimension and $L$ is PGF, such
      that $[q] = [0] \in \mathfrak{FF}\mbox{-}\mbox{Hom}_{R}(L,M)$.
\end{itemize}
\end{Corollary}

\begin{proof}
The implications (i)$\Rightarrow$(ii) and (ii)$\Rightarrow$(iii) are clear
and hence it suffices to show that (iii)$\Rightarrow$(i). To that end, we
note that the adjunction isomorphism established by Theorem \ref{thr1}(ii)
implies, in particular, that the additive map
\[ [q]_{*} : \mathfrak{FF}\mbox{-}\mbox{Hom}_{R}(L,L) \rightarrow
             \mathfrak{FF}\mbox{-}\mbox{Hom}_{R}(L,M) \]
is bijective. Since
$[q]_{*}[1_{L}] = [q][1_{L}] = [q1_{L}] = [q] = [0] \in
 \mathfrak{FF}\mbox{-}\mbox{Hom}_{R}(L,M)$,
we conclude that $[1_{L}] = [0] \in \mathfrak{FF}\mbox{-}\mbox{Hom}_{R}(L,L)$;
hence, the identity map of $L$ factors through a module $F'$ of finite flat dimension. It follows that $L$ is a direct summand of $F'$ and hence $L$ has finite flat dimension. Since $L$ is PGF and hence Gorenstein flat, Lemma \ref{lemataki} implies that the $R$-module $L$ is flat. Using the short exact sequence $0 \rightarrow F \rightarrow L \rightarrow M \rightarrow 0$ we obtain that $\textrm{fd}_R M<\infty$.
\end{proof}

We consider now the class of flat modules ${\tt Flat}(R)$ and we note that the set consisting of those $R$-linear maps
$f : M \rightarrow N$ that factor through a flat module is a subgroup of the abelian group $\mbox{Hom}_{R}(M,N)$. We denote by $\mathfrak{F}\mbox{-}\mbox{Hom}_{R}(M,N)$ the corresponding quotient
group and by $[f]_{\mathfrak{F}}$ the class of any $f \in \mbox{Hom}_{R}(M,N)$. Moreover, we denote by $\mathfrak{F}\mbox{-}R\textrm{-Mod}$ the category whose objects are all $R$-modules and whose morphism sets are given by the abelian groups $\mathfrak{F}\mbox{-}\mbox{Hom}_{R}(M,N)$. Using the approximation sequences (\ref{eq3}),(\ref{eq4}) and similar arguments the following result is obtained immediately. Although, we present an alternative proof. 

\begin{Proposition}\label{cor28}The following conditions are equivalent for a Gorenstein flat $R$-module $M$.
\begin{itemize}
	\item[(i)]$M$ is flat,
	\item[(ii)]$\mathfrak{F}\mbox{-}\mbox{Hom}_{R}(L,M) = 0$ for every PGF $R$-module $L$ and
	\item[(iii)] there exists a short exact sequence of $R$-modules
	$0 \rightarrow F \rightarrow L \stackrel{q}{\rightarrow}
	M \rightarrow 0$,
	where $F$ is flat and $L$ is PGF, such
	that $[q] = [0] \in \mathfrak{F}\mbox{-}\mbox{Hom}_{R}(L,M)$.
\end{itemize}
\end{Proposition}
%\[\begin{tikzcd}[ampersand replacement=\&]
	%	\&\& 0 \& 0 \\
	%	0 \& F \& L \& M \& 0 \\
	%	0 \& F \& B \& {F'} \& 0 \\
	%	\&\& C \& C \\
	%	\&\& 0 \& 0
	%	\arrow[from=2-2, to=2-3]
	%	\arrow["q", from=2-3, to=2-4]
	%	\arrow[from=2-4, to=2-5]
	%	\arrow[from=2-4, to=1-4]
	%	\arrow[from=2-3, to=1-3]
	%	\arrow[from=2-1, to=2-2]
	%	\arrow[from=3-1, to=3-2]
	%	\arrow[from=3-2, to=3-3]
	%	\arrow[dashed, from=3-3, to=3-4]
	%	\arrow[from=3-4, to=3-5]
	%	\arrow["id", tail, two heads, from=3-2, to=2-2]
	%	\arrow["\pi"', dashed, from=3-3, to=2-3]
	%	\arrow["f"', from=3-4, to=2-4]
	%	\arrow[from=4-3, to=3-3]
	%	\arrow[from=4-4, to=3-4]
	%	\arrow["id"', tail, two heads, from=4-3, to=4-4]
	%	\arrow[from=5-3, to=4-3]
	%	\arrow[from=5-4, to=4-4]
	%\end{tikzcd}\]
\begin{proof} Since the implications $(i)\Rightarrow(ii)$ and $(ii)\Rightarrow(iii)$ are clear, it suffices to show that $(iii)\Rightarrow(i)$. We consider a special flat precover $(F',f)$ of $M$ and the following pull-back diagram: 
\[
\begin{array}{ccccccccc}
	& & & & 0 & & 0 & &\\
	& & & & \uparrow & & \uparrow & &\\
	0&\rightarrow & F &\rightarrow & L &\xrightarrow{q} & M & \rightarrow &0\\
	& & \parallel & & \!\!\!\! {\scriptstyle{\pi}}\uparrow & & \uparrow {\scriptstyle{f}}\!\!\!& &\\
	0&\rightarrow & F &\rightarrow & B &\rightarrow & F' & \rightarrow &0\\
	& & & & \uparrow & & \uparrow & &\\
	& & & & C & = & C & &\\
	& & & & \uparrow & & \uparrow & &\\
	& & & & 0 & & 0 & &
\end{array}
\]

\noindent Since $q$ factors through a flat module we can write $q$ as the composition $L\rightarrow \Lambda\rightarrow M$ where $\Lambda$ is a flat module. Thus, $q$ factors through the special flat precover $(F',f)$ as a composition of the form \[L\rightarrow\Lambda\rightarrow F'\xrightarrow{f} M.\] The universal property of the pull-back yields the existence of a morphism $\Lambda\rightarrow B$ and hence $\pi$ is $R$-split. Moreover, the $R$-module $L$ is flat as direct summand of the flat $R$-module $B$, implying that $M$ is a Gorenstein flat $R$-module. We conclude that the $R$-module M is flat (see Lemma \ref{lemataki}).  
\end{proof}

\begin{Remark}\rm The proof of Proposition 2.8 shows that if an epimorphism with a flat kernel factors through a flat module, then the flat dimension of its cokernel is at most one.\end{Remark}

\section{Stabilizing with respect to PGF modules}
Given two modules $M,N$, the set consisting of those $R$-linear maps
$f : M \rightarrow N$ that factor through a PGF module
is a subgroup of the abelian group $\mbox{Hom}_{R}(M,N)$. We shall denote by $\mathfrak{PGF}\mbox{-}\mbox{Hom}_{R}(M,N)$ the
corresponding quotient group and let $[f] = [f]_{\mathfrak{PGF}}$ be the class
of any $R$-linear map $f \in \mbox{Hom}_{R}(M,N)$ therein. We note that for any 
three modules $M,N$ and $L$ the composition of $R$-linear maps
induces a well-defined biadditive map
\[ \mathfrak{PGF}\mbox{-}\mbox{Hom}_{R}(N,L) \times
   \mathfrak{PGF}\mbox{-}\mbox{Hom}_{R}(M,N)
   \rightarrow \mathfrak{PGF}\mbox{-}\mbox{Hom}_{R}(M,L) . \]
We shall denote by $\mathfrak{PGF}\mbox{-}R\mbox{-Mod}$ the category whose objects
are all modules and whose morphism sets are given by the abelian groups
$\mathfrak{PGF}\mbox{-}\mbox{Hom}_{R}(M,N)$ (with composition of morphisms induced
by the composition of $R$-linear maps). Then, Lemma \ref{Scanuel} implies that the module
$F'$, which is defined for every module $M$ of finite Gorenstein flat dimension by (\ref{eq6}), is uniquely determined up to isomorphism as an object of the category $\mathfrak{PGF}\mbox{-}R\mbox{-Mod}$.

\begin{Lemma}\label{lem31}
Let $f : M \rightarrow N$ be an $R$-linear map, where $M,N$ are two modules of finite Gorenstein flat dimension. We also consider two exact sequences
of $R$-modules of the form
\[ 0 \rightarrow M \stackrel{\iota}{\rightarrow} F
     \stackrel{p}{\rightarrow} L \rightarrow 0
   \;\;\; \mbox{and} \;\;\;
   0 \rightarrow N \stackrel{\jmath}{\rightarrow} F'
     \stackrel{q}{\rightarrow} L' \rightarrow 0 , \]
where $F,F'$ have finite flat dimension and $L,L'$ are PGF. Then,
\begin{itemize}
\item[(i)] There exists an $R$-linear map $g : F \rightarrow F'$, such that
$g \iota = \jmath f$.

\item[(ii)] If $g,g' : F \rightarrow F'$ are two $R$-linear maps with
$g \iota = \jmath f$ and $g' \iota = \jmath f$, then
$[g] = [g'] \in \mathfrak{PGF}\mbox{-}\mbox{Hom}_{R}(F,F')$.

\item[(iii)] If $[f] = [0] \in \mathfrak{PGF}\mbox{-}\mbox{Hom}_{R}(M,N)$ and
$g : F \rightarrow F'$ is an $R$-linear map such that
$g \iota = \jmath f$, then $[g] = [0] \in \mathfrak{PGF}\mbox{-}\mbox{Hom}_{R}(F,F')$.
\end{itemize}
\end{Lemma}

\begin{proof}
(i) Since $\mbox{Ext}_{R}^{1} \! \left( L,F' \right)=0$, the additive map
$\iota^{*} : \mbox{Hom}_{R}(F,F') \rightarrow \mbox{Hom}_{R}(M,F')$
is surjective. Therefore, there exists an $R$-linear map
$g : F \rightarrow F'$ such that $\jmath f = \iota^{*}(g)$, as
needed.

(ii) Let $g,g' : F \rightarrow F'$ be two $R$-linear maps with
$g \iota = \jmath f$ and $g' \iota = \jmath f$.
\[
\begin{array}{ccccccccc}
 0 & \rightarrow & M & \stackrel{\iota}{\rightarrow} & F
   & \stackrel{p}{\rightarrow} & L & \rightarrow & 0 \\
 & & \!\!\! {\scriptstyle{f}} \downarrow & & {\scriptstyle{g}} \downarrow
     \downarrow {\scriptstyle{g'}} & & & & \\
 0 & \rightarrow & N & \stackrel{\jmath}{\rightarrow} & F'
   & \stackrel{q}{\rightarrow} & L' & \rightarrow & 0
\end{array}
\]
Then, $(g'-g)\iota = g'\iota - g\iota = \jmath f - \jmath f = 0$ and hence
there exists an $R$-linear map $h : L \rightarrow F'$, such that
$g'-g = hp$. Since the $R$-module $L$ is PGF, we conclude that
$[g] = [g'] \in \mathfrak{PGF}\mbox{-}\mbox{Hom}_{R}(F,F')$.

(iii) Assume that $f$ factors as the composition of two $R$-linear maps
$M \stackrel{a}{\rightarrow} \Gamma \stackrel{b}{\rightarrow} N$,
where the module $\Gamma$ is PGF. The definition of PGF modules yields the existence of a short exact sequence of $R$-modules of the form
\[ 0 \rightarrow \Gamma \stackrel{k}{\rightarrow} P
     \rightarrow \Gamma' \rightarrow 0, \]
where $P$ is projective and
$\Gamma'$ is PGF (see also Proposition \ref{pgf}). Then, invoking (i) above, we may find $R$-linear
maps $\alpha : F \rightarrow P$ and $\beta : P \rightarrow F'$, such
that $\alpha \iota = ka$ and $\beta k = \jmath b$.
\[
\begin{array}{ccccccccc}
 0 & \rightarrow & M & \stackrel{\iota}{\rightarrow} & F
   & \stackrel{p}{\rightarrow} & L & \rightarrow & 0 \\
 & & \!\!\! {\scriptstyle{a}} \downarrow & &
     \!\!\! {\scriptstyle{\alpha}} \downarrow & & & & \\
 0 & \rightarrow & \Gamma & \stackrel{k}{\rightarrow} & P &
     \rightarrow & \Gamma' & \rightarrow & 0 \\
 & & \!\!\! {\scriptstyle{b}} \downarrow & &
     \!\!\! {\scriptstyle{\beta}} \downarrow & & & & \\
 0 & \rightarrow & N & \stackrel{\jmath}{\rightarrow} & F'
   & \stackrel{q}{\rightarrow} & L' & \rightarrow & 0
\end{array}
\]
We conclude that $(\beta \alpha) \iota = \jmath (ba) = \jmath f$ and hence
for every $R$-linear map $g : F \rightarrow F'$ with $g \iota = \jmath f$
we have $[g] = [\beta \alpha] \in \mathfrak{PGF}\mbox{-}\mbox{Hom}_{R}(F,F')$
(see (ii) above). This finishes the proof, since we obviously have
$[\beta \alpha] = [0] \in \mathfrak{PGF}\mbox{-}\mbox{Hom}_{R}(F,F')$.\end{proof}

In the proof of Lemma \ref{lem31}, the fact that the modules $F,F'$ have finite flat dimension used only in (i) to obtain the vanishing of the group $\textrm{Ext}^1_R(L,F')$. We may relax the condition of the flatness of $F,F'$ as it is stated in the following lemma.

\begin{Lemma}\label{lem32}
Let $f : M \rightarrow N$ be an $R$-linear map, where $M,N$ are two  
$R$-modules and consider two short exact sequences
\[ 0 \rightarrow M \stackrel{\iota}{\rightarrow} K
     \stackrel{p}{\rightarrow}L \rightarrow 0
   \;\;\; \mbox{and} \;\;\;
   0 \rightarrow N \stackrel{\jmath}{\rightarrow} K'
     \stackrel{q}{\rightarrow} L' \rightarrow 0 , \]
where $K,K'\in{\tt PGF}(R)^\perp$ and $L,L'\in{\tt PGF}(R)$. Then,
\begin{itemize}
\item[(i)] There exists an $R$-linear map $g : K \rightarrow K'$, such that
$g \iota = \jmath f$.

\item[(ii)] If $g,g' : K \rightarrow K'$ are two $R$-linear maps with
$g \iota = \jmath f$ and $g' \iota = \jmath f$, then
$[g] = [g'] \in \mathfrak{PGF}\mbox{-}\mbox{Hom}_{R}(K,K')$.

\item[(iii)] If $[f] = [0] \in \mathfrak{PGF}\mbox{-}\mbox{Hom}_{R}(M,N)$ and
$g : K \rightarrow K'$ is an $R$-linear map such that
$g \iota = \jmath f$, then $[g] = [0] \in \mathfrak{PGF}\mbox{-}\mbox{Hom}_{R}(K,K')$.
\end{itemize}
\end{Lemma}
\begin{proof} Same as the proof of Lemma \ref{lem31}.\end{proof}

As in Section 3 there are well defined additive functors 
\[ \nu : \mathfrak{PGF}\mbox{-}\overline{{\tt GFlat}}(R)\rightarrow
         \mathfrak{PGF}\mbox{-}\overline{{\tt Flat}}(R) , \]
\[  \nu ' : \mathfrak{PGF}\mbox{-}R\mbox{-Mod}\rightarrow
         \mathfrak{PGF}\mbox{-}{\tt PGF}(R)^\perp,\]
where $\nu$ maps an $R$-module $M$ of finite Gorenstein flat dimension to the $R$-module $F'$ of finite flat dimension that appears in the short exact sequence (\ref{eq6}) and $\nu'$ maps an $R$-module $M$ to a module in ${\tt PGF}(R)^\perp$ that appears in an approximation sequence of $M$ of the form
\[ 0 \rightarrow M \rightarrow K
     \rightarrow L \rightarrow 0\]
with $K\in{\tt PGF}(R)^\perp$ and $L\in{\tt PGF}(R)$.
Moreover, $\nu$ maps the class $[f]\in \mathfrak{PGF}\mbox{-}\mbox{Hom}_R(M,N)$ of an $R$-linear map $f: M \rightarrow N$ to the class $[g]\in \mathfrak{PGF}\mbox{-}\mbox{Hom}_{R}(F,F')$ of any $R$-linear map $g: F \rightarrow F'$ which satisfies the condition described in Lemma \ref{lem31}(i), while $\nu '$ maps the class $[f]\in \mathfrak{PGF}\mbox{-}\mbox{Hom}_R(M,N)$ of an $R$-linear map $f: M \rightarrow N$ to the class $[g]\in \mathfrak{PGF}\mbox{-}\mbox{Hom}_{R}(K,K')$ of any $R$-linear map $g: K \rightarrow K'$ which satisfies the condition described in Lemma \ref{lem32}(i). 

The main result of this section is stated below.

\begin{Theorem}\label{theo2}
(i) The additive functor
$\nu : \mathfrak{PGF}\mbox{-}\overline{{\tt GFlat}}(R)\rightarrow
         \mathfrak{PGF}\mbox{-}\overline{{\tt Flat}}(R)$
defined above is left adjoint to the inclusion functor
$\mathfrak{PGF}\mbox{-}\overline{{\tt Flat}}(R) \hookrightarrow \mathfrak{PGF}\mbox{-}\overline{{\tt GFlat}}(R)$.

(ii) The additive functor $ \nu ' : \mathfrak{PGF}\mbox{-}R\mbox{-Mod}\rightarrow\mathfrak{PGF}\mbox{-}{\tt PGF}(R)^\perp$ is left adjoint to the inclusion functor $\mathfrak{PGF}\mbox{-}{\tt PGF}(R)^\perp \hookrightarrow \mathfrak{PGF}\mbox{-}R\mbox{-Mod}$.
\end{Theorem}
\begin{proof} We prove only (i) as the proof of (ii) is essentially the same. Let $F$ be an $R$-module of finite flat dimension and $M$ be an $R$-module of finite Gorenstein flat dimension. We also consider a short exact
sequence of $R$-modules
\[ 0 \rightarrow M \stackrel{\iota}{\rightarrow} F'
     \stackrel{p}{\rightarrow} L \rightarrow 0 , \]
where $F'$ has finite flat dimension and $L$ is PGF. We note that the additive map
\[
 [\iota]^{*} : \mathfrak{PGF}\mbox{-}\mbox{Hom}_{R}(F',F) \rightarrow
 \mathfrak{PGF}\mbox{-}\mbox{Hom}_{R}(M,F)
\]
is natural in both $F$ (this is straightforward) and $M$ (this follows from
Lemma 4.1). We shall establish the adjunction in the statement of the theorem,
by proving that the additive map above is bijective. Indeed, since $L$ is a PGF module and $F$ has finite flat dimension, the group
$\mbox{Ext}_{R}^{1} \! \left( L,F \right)$ is trivial. Thus, the additive map
\[ \iota^{*} : \mbox{Hom}_{R}(F',F) \rightarrow \mbox{Hom}_{R}(M,F) \]
is surjective and hence $[\iota]^{*}$ is also surjective. We assume now that $g : F \rightarrow F'$ is an $R$-linear map such that $[g\iota] = [g][\iota] = [\iota]^{*}[g] = [0] \in
 \mathfrak{PGF}\mbox{-}\mbox{Hom}_{R}(M,F)$
and consider the following commutative diagram:
\[
\begin{array}{ccccccccc}
 0 & \rightarrow & M & \stackrel{\iota}{\rightarrow} & F'
   & \stackrel{p}{\rightarrow} & L & \rightarrow & 0 \\
 & & \!\!\!\!\! {\scriptstyle{g \iota}} \downarrow & &
     \!\!\! {\scriptstyle{g}} \downarrow & & & & \\
 0 & \rightarrow & F & = & F & \rightarrow & 0 & \rightarrow & 0
\end{array}
\]
Invoking Lemma 3.1(iii), we obtain that
$[g] = [0] \in \mathfrak{PGF}\mbox{-}\mbox{Hom}_{R}(F',F)$ and hence $[\iota]^{*}$ is injective.\end{proof}

\begin{Corollary}
The following conditions are equivalent for an $R$-module $M$ of finite Gorenstein flat dimension.
\begin{itemize}
\item[(i)] $M$ is PGF,

\item[(ii)] $\mathfrak{PGF}\mbox{-}\mbox{Hom}_{R}(M,F) = 0$ for every $R$-module $F$ of finite flat dimension,

\item[(iii)] there exists a short exact sequence of $R$-modules
      $0 \rightarrow M \stackrel{\iota}{\rightarrow} F
         \rightarrow L \rightarrow 0$,
      where $F$ has finite flat dimension and $L$ is PGF,
      such that $[\iota] = [0] \in \mathfrak{PGF}\mbox{-}\mbox{Hom}_{R}(M,F)$.
\end{itemize}
\end{Corollary}
\begin{proof}
The implications (i)$\Rightarrow$(ii) and (ii)$\Rightarrow$(iii) are obvious
and hence it suffices to show that (iii)$\Rightarrow$(i). To that end, we
note that the adjunction isomorphism of the previous Theorem
implies, in particular, that the additive map
\[ [\iota]^{*} : \mathfrak{PGF}\mbox{-}\mbox{Hom}_{R}(F,F) \rightarrow
                 \mathfrak{PGF}\mbox{-}\mbox{Hom}_{R}(M,F) \]
is bijective. Since
$[\iota]^{*}[1_{F}] = [1_{F}][\iota] = [1_{F}\iota] = [\iota] = [0] \in
 \mathfrak{PGF}\mbox{-}\mbox{Hom}_{R}(M,F)$,
we conclude that $[1_{F}] = [0] \in \mathfrak{PGF}\mbox{-}\mbox{Hom}_{R}(F,F)$;
hence, the identity map of $F$ factors through a PGF module
$L$. It follows that $F$ is a direct summand of $L$ and hence $F$ is PGF. Being a kernel of an epimorphism of PGF modules, $N$ is necessarily PGF. \end{proof}

\begin{Corollary}
The following conditions are equivalent for an $R$-module $M$.
\begin{itemize}
\item[(i)] $M$ is PGF,

\item[(ii)] $\mathfrak{PGF}\mbox{-}\mbox{Hom}_{R}(M,K) = 0$ for every $K\in{\tt PGF}(R)^\perp$,

\item[(iii)] there exists a short exact sequence of $R$-modules
    $0 \rightarrow M \stackrel{\iota}{\rightarrow} K
         \rightarrow L \rightarrow 0$,
  where $K\in{\tt PGF}(R)^\perp$ and $L$ is PGF,
 such that $[\iota] = [0] \in \mathfrak{PGF}\mbox{-}\mbox{Hom}_{R}(M,K)$.
\end{itemize}
\end{Corollary}

\begin{proof} Same as the proof of Corollary 4.4.\end{proof}

\section{Another stability of PGF modules}
 This final section proves that the class of PGF $R$-modules coincides with the class ${\tt PGF}^{(2)}_{\pazocal{I}}(R)$, where we denote by ${\tt PGF}^{(2)}_{\pazocal{I}}(R)$ the class of $R$-modules $M$ which are syzygies in an acyclic complex $\textbf{G}$ of PGF modules, such that $I\otimes_R\textbf{G}$ is acyclic for every injective $R$-module $I$. Also, we denote by ${\tt {SPGF}}^{(2)}_{\pazocal{I}}(R)$ the subcategory of the R-modules $M$ for which there exists a short exact sequence of $R$-modules of the form $0\rightarrow M \rightarrow G \rightarrow M \rightarrow 0$, where $G$ is a PGF, such that $I\otimes_R \_\!\_$ preserves the exactness of this sequence whenever $I$ is an injective $R$-module.

%(respectively, ${\tt PGF}^{(2)}(R)$)(respectively, $H\otimes_R\_\!\_$)  (respectively, $H$ is a Gorenstein injective $R$-module) Since every injective module is Gorenstein injective, we have ${\tt {PGF}}(R)\subseteq {\tt {PGF}}^{(2)}(R) \subseteq {\tt {PGF}}^{(2)}_{\pazocal{I}}(R)$ and hence it suffices to prove the inclusion ${\tt {PGF}}^{(2)}_{\pazocal{I}}(R)\subseteq {\tt {PGF}}(R)$.
\begin{Proposition}\label{prop44}
	Let $M$ be an $R$-module in the class ${\tt{PGF}}^{(2)}_{\pazocal{I}}(R)$. Then, $\textrm{Tor}_i^R(I,M)=0$ for every $i>0$ and every injective $R$-module $I$.
\end{Proposition}

\begin{proof} Let $M\in {\tt PGF}^{(2)}_{\pazocal{I}}(R)$ and $I$ be an injective $R$-module. Then, there exists an exact sequence of PGF $R$-modules $$\cdots \rightarrow G_1 \rightarrow G_0 \rightarrow G^0 \rightarrow G^1 \rightarrow \cdots,$$ where $M\cong \textrm{Im}(G_0 \rightarrow G^0)$ and the functor $J\otimes_R \_\!\_ $ preserves the exactness of this sequence for every injective $R$-module $J$. We consider a short exact sequence of $R$-modules of the form $0\rightarrow K \rightarrow G_0 \rightarrow M \rightarrow 0$, where $K=\textrm{Im}(G_1\rightarrow G_0)\in {\tt PGF}^{(2)}_{\pazocal{I}}(R)$, and obtain the short exact sequence $0\rightarrow I\otimes_R K \rightarrow I\otimes_R G_0 \rightarrow I\otimes_R M \rightarrow 0$. Since the module $G_0$ is PGF, Proposition \ref{pgf}(ii) yields $\textrm{Tor}_i^R(I,G_0)=0$ for every $i>0$. Then, the short exact sequence $0\rightarrow K \rightarrow G_0 \rightarrow M \rightarrow 0$ induces a long exact sequence of the form $$\cdots \rightarrow \textrm{Tor}_1^R(I,G_0) \rightarrow \textrm{Tor}_1^R(I,M)\rightarrow I\otimes_R K \rightarrow I\otimes_R G_0 \rightarrow I\otimes_R M \rightarrow 0, $$ which implies that $\textrm{Tor}_1^R(I,M)=0$. 
	Moreover, the long exact sequence $$\cdots\rightarrow \textrm{Tor}_{i+1}^R(I,G_0)\rightarrow \textrm{Tor}_{i+1}^R(I,M) \rightarrow \textrm{Tor}_{i}^R(I,K)\rightarrow \textrm{Tor}_{i}^R(I,G_0) \rightarrow \cdots,$$ where $i>0$, yields $\textrm{Tor}_{i+1}^R(I,M) = \textrm{Tor}_{i}^R(I,K)$ for every $i>0$. Thus, using induction on $i$ and the fact that $K$ lies in ${\tt PGF}^{(2)}_{\pazocal{I}}(R)$, we obtain that $\textrm{Tor}_i^R(I,M)=0$ for every $i>0$ and every injective $R$-module $I$. \end{proof}
	
	%We suppose now that $n\geq 1$. Let $0\rightarrow I' \rightarrow I_0 \rightarrow I_1 \rightarrow \cdots \rightarrow I_n \rightarrow 0$ be an injective resolution of $I'$ of length $n$ and $J=\textrm{Im}(I_0 \rightarrow I_1)$. Since $\textrm{id}_R (J)\leq n-1$, our inductive hypothesis implies that $\textrm{Tor}_i^R(J,M)=0$ for every $i>0$. Thus, the short exact sequence $0\rightarrow I' \rightarrow I_0 \rightarrow J \rightarrow 0$ induces a long exact sequence of the form $$ \cdots \rightarrow \textrm{Tor}_{i+1}^R(J,M)\rightarrow \textrm{Tor}_{i}^R(I',M) \rightarrow \textrm{Tor}_{i}^R(I_0,M) \rightarrow \cdots,$$ where $i>0$, from which we obtain that $\textrm{Tor}_{i}^R(I',M)=0$ for every $i>0$.

The following proposition gives a characterization of the subcategory ${\tt SPGF}^{(2)}_{\pazocal{I}}(R)$.

\begin{Proposition}\label{props2pgfi}
	The following conditions are equivalent for an $R$-module $M$:
	\begin{itemize}
		\item[(i)] $M \in {\tt SPGF}^{(2)}_{\pazocal{I}}(R)$.
		\item[(ii)] There exists a short exact sequence of $R$-modules of the form $0\rightarrow M \rightarrow G \rightarrow M \rightarrow 0$, where G is PGF and $\textrm{Tor}_1^R(I,M)=0$ for every injective $R$-module $I$.
		\item[(iii)] There exists a short exact sequence of $R$-modules of the form $0\rightarrow M \rightarrow G \rightarrow M \rightarrow 0$, where G is PGF and $\textrm{Tor}_i^R(I,M)=0$ for every $i>0$ and every injective $R$-module $I$.
		%\item[(iv)] There exists a short exact sequence of $R$-modules of the form $0\rightarrow M \rightarrow G \rightarrow M \rightarrow 0$, where G is PGF and $\textrm{Tor}_i^R(I',M)=0$ for every $i>0$ and every $R$-module $I'$ with finite injective dimension.
		%\item[(v)] There exists a short exact sequence of $R$-modules of the form $0\rightarrow M \rightarrow G \rightarrow M \rightarrow 0$, where G is PGF and $I' \otimes_R \_\!\_$ preserves the exactness of this sequence for every $R$-module $I'$ with finite injective dimension.	
	\end{itemize}
\end{Proposition}

\begin{proof}
	This follows immediately from the definition of the class ${\tt SPGF}^{(2)}_{\pazocal{I}}(R)$ and Proposition \ref{prop44}, using standard arguments. 
\end{proof}

\begin{Proposition}\label{final3}
	Every module in ${\tt PGF}^{(2)}_{\pazocal{I}}(R)$ is a direct summand of a module in ${\tt SPGF}^{(2)}_{\pazocal{I}}(R)$.
\end{Proposition}

\begin{proof}
	Let $M$ be an $R$-module in ${\tt PGF}^{(2)}_{\pazocal{I}}(R)$. Then, there exists an exact sequence of PGF $R$-modules 
	$$ \textbf{G}= \cdots \rightarrow G_1 \xrightarrow{d_1^G} G_0 \xrightarrow{d_0^G} G_{-1} \xrightarrow{d_{-1}^G} G_{-2} \rightarrow \cdots,$$
	such that $M\cong \textrm{Im}(d_0^G)$ and such that the sequence $I\otimes_R \textbf{G}$ is exact for every injective module $I$. For every $n \in \mathbb{Z}$, we denote by $\Sigma^n\textbf{G}$ the exact sequence obtained from $\textbf{G}$ by increasing all indices by n: $(\Sigma^n\textbf{G})_i=G_{i-n}$ and $d_i^{\Sigma^n G}=d_{i-n}^G$ for every $i\in \mathbb{Z}$. Consider now the exact sequence
	$$\bigoplus_{n \in \mathbb{Z}}(\Sigma^n\textbf{G})= \cdots \rightarrow \bigoplus_{i \in \mathbb{Z}}G_i \xrightarrow{\bigoplus_{i \in \mathbb{Z}}d_i^G}\bigoplus_{i \in \mathbb{Z}}G_i \xrightarrow{\bigoplus_{i \in \mathbb{Z}}d_i^G}\bigoplus_{i \in \mathbb{Z}}G_i \rightarrow \cdots.$$ Since ${\tt PGF}(R)$ is closed under direct sums, we obtain that the module $\bigoplus_{i \in \mathbb{Z}}G_i$ is also PGF. Moreover, \cite[Proposition 20.2(3)]{AF} yields the isomorphism of complexes 
	$I\otimes_R (\bigoplus_{n \in \mathbb{Z}}(\Sigma^n\textbf{G})) \cong \bigoplus_{n \in \mathbb{Z}}(I\otimes_R \Sigma^n\textbf{G})$ which is an exact sequence for every injective module $I$. Thus, $\textrm{Im}(\bigoplus_{i \in \mathbb{Z}}d_i^G)$ lies in ${\tt SPGF}^{(2)}_{\pazocal{I}}(R)$, as needed. \end{proof}
	%and $M$ is a direct summand of this module.

\begin{Definition}
	Let $M$ be a module in ${\tt SPGF}^{(2)}_{\pazocal{I}}(R)$. We say that a module N is an $M$-type if there exists a short exact sequence of $R$-modules of the form $0\rightarrow M \rightarrow N \rightarrow G \rightarrow 0,$ where $G$ is PGF. 
\end{Definition}

\begin{Proposition}\label{final2}
	Let $M$ be an $R$-module in ${\tt SPGF}^{(2)}_{\pazocal{I}}(R)$ and $N$ be an $M$-type module. Then, the following hold:
	\begin{itemize}
		\item[(i)] $\textrm{Tor}_i^R(I,N)=0$ for every $i>0$ and every injective $R$-module $I$.
		\item[(ii)] There exists an exact sequence of $R$-modules of the form $0\rightarrow N \rightarrow P \rightarrow K \rightarrow 0$, where $P$ is projective, $K$ is an $M$-type module and the functor $I\otimes_R \_\!\_$ preserves the exactness of this sequence for every injective $R$-module $I$.
	\end{itemize}
	
\end{Proposition}

\begin{proof}
	(i) Let $I$ be an injective $R$-module. Since $N$ is an $M$-type, there exists a short exact sequence of $R$-modules of the form $0\rightarrow M \rightarrow N \rightarrow G \rightarrow 0$, where $G$ is PGF. The short exact sequence above induces a long exact sequence of the form $$\cdots \rightarrow \textrm{Tor}_{i}^R(I,M) \rightarrow \textrm{Tor}_{i}^R(I,N) \rightarrow \textrm{Tor}_{i}^R(I,G) \rightarrow \cdots,$$ where $i>0$. Since $M\in {\tt SPGF}^{(2)}_{\pazocal{I}}(R)$, by Proposition \ref{props2pgfi}(iii), we have $\textrm{Tor}_{i}^R(I,M)=0$ for every $i>0$. Moreover, Proposition \ref{pgf}(ii) yields $\textrm{Tor}_{i}^R(I,G)=0$ for every $i>0$. We conclude that $\textrm{Tor}_{i}^R(I,N)=0$ for every $i>0$ and every injective $R$-module $I$.
	
	(ii) Since $M\in {\tt SPGF}^{(2)}_{\pazocal{I}}(R)$, there exists a short exact sequence of $R$-modules of the form $0\rightarrow M \rightarrow G' \rightarrow M \rightarrow 0$, where $G'$ is PGF. Since $N$ is an $M$-type, there exists also a short exact sequence of $R$-modules of the form $0\rightarrow M \rightarrow N \rightarrow G \rightarrow 0$, where $G$ is PGF. Consider the pushout diagram of the above short exact sequences:
	\[
	\begin{array}{ccccccccc}
		& & 0 & & 0 & & & &\\
		& & \downarrow & & \downarrow & & & &\\
		0&\rightarrow & M &\rightarrow & G' &\rightarrow & M & \rightarrow &0\\
		& & \downarrow & & \downarrow & & \parallel & &\\
		0&\rightarrow & N &\rightarrow & F &\rightarrow & M & \rightarrow &0\\
		& & \downarrow & & \downarrow & & & &\\
		& & G & = & G & & & &\\
		& & \downarrow & & \downarrow & & & &\\
		& & 0 & & 0 & & & &
	\end{array}
	\]
	\noindent Since ${\tt PGF}(R)$ is closed under extensions, the short exact sequence $0\rightarrow G' \rightarrow F \rightarrow G \rightarrow 0$ implies that the module $F$ is also PGF. Thus, there exists a short exact sequence of $R$-modules of the form $0\rightarrow F \rightarrow P \rightarrow F' \rightarrow 0$, where $P$ is projective and $F'$ is PGF. Consider now the following pushout diagram:
	\[
	\begin{array}{ccccccccc}
		& & & & 0 & & 0 & &\\
		& & & & \downarrow & & \downarrow & &\\
		0&\rightarrow & N &\rightarrow & F &\rightarrow & M & \rightarrow &0\\
		& & \parallel & & \downarrow & & \downarrow & &\\
		0&\rightarrow & N &\rightarrow & P &\rightarrow & K & \rightarrow &0\\
		& & & & \downarrow & & \downarrow & &\\
		& & & & F' & = & F' & &\\
		& & & & \downarrow & & \downarrow & &\\
		& & & & 0 & & 0 & &
	\end{array}
	\]
	\noindent Since $F'$ is PGF, the module $K$ is an $M$-type. Moreover, (i) implies $\textrm{Tor}_1^R(I,K)=0$ for every injective $R$-module $I$ and hence the sequence $0\rightarrow I\otimes_R N \rightarrow I\otimes_R P \rightarrow I\otimes_R K\rightarrow 0$ is exact for every injective $R$-module $I$.\end{proof}

\begin{Corollary}\label{corfinal}
	Let $M$ be an $R$-module in ${\tt SPGF}^{(2)}_{\pazocal{I}}(R)$ and $N$ be an $M$-type module. Then, $N$ is PGF.
\end{Corollary}

\begin{proof}
	Since $N$ is an $M$-type module, Proposition \ref{final2}(ii) implies that there exists a short exact sequence of $R$-modules of the form $0\rightarrow N \rightarrow P^0 \rightarrow K \rightarrow 0$, where $P^0$ is projective, $K$ is an $M$-type and the functor $I\otimes_R \_\!\_$ preserves the exactness of this sequence for every injective $R$-module $I$. The iteration of this process yields an exact sequence of $R$-modules of the form $0\rightarrow N \rightarrow P^0 \rightarrow P^1 \rightarrow P^2 \rightarrow \cdots,$ where $P^i$ is projective for every $i\geq 0$ and the functor $I\otimes_R \_\!\_$ preserves the exactness of this sequence for every injective $R$-module $I$. Using Proposition \ref{final2}(i), we also have $\textrm{Tor}_{i}^R(I,N)=0$ for every $i>0$ and every injective $R$-module $I$. Invoking Proposition \ref{pgf}(ii), we conclude that $N$ is PGF.
\end{proof}

\begin{Theorem} \label{finale}
	${\tt PGF}(R)= {\tt {PGF}}^{(2)}_{\pazocal{I}}(R).$
\end{Theorem}

\begin{proof}
	It suffices to prove that ${\tt PGF}^{(2)}_{\pazocal{I}}(R)\subseteq {\tt PGF}(R)$. Since the class ${\tt PGF}(R)$ is closed under direct summands, by Proposition \ref{final3} it suffices to prove that ${\tt SPGF}^{(2)}_{\pazocal{I}}(R)\subseteq {\tt PGF}(R)$. Let $M$ be an $R$-module in ${\tt SPGF}^{(2)}_{\pazocal{I}}(R)$. Letting $G=0$ in Definition 5.4, it follows that $M$ is an $M$-type. Thus, Corollary 5.6 implies that $M$ is PGF. \end{proof}
\begin{Remark}\rm (i) We denote by ${\tt PGF}^{(2)}(R)$ the subcategory of the R-modules $M$ which are syzygies in an acyclic complex $\textbf{G}$ of PGF modules, such that $H\otimes_R\textbf{G}$ is acyclic for every Gorenstein injective $R$-module $H$. Then we have the inclusions ${\tt {PGF}}(R)\subseteq {\tt {PGF}}^{(2)}(R) \subseteq {\tt {PGF}}^{(2)}_{\pazocal{I}}(R)$ and hence Theorem 5.7 yields ${\tt PGF}(R)= {\tt {PGF}}^{(2)}(R)={\tt {PGF}}^{(2)}_{\pazocal{I}}(R)$.
	
(ii)The stability result in this section may be transfered to the relative case under the assumption that $({\tt PGF}_{\mathcal{B}}(R),{\tt PGF}_{\mathcal{B}}^\perp (R))$ is a complete hereditary cotorsion pair.	More precicely we have ${\tt PGF}_{\mathcal{B}}(R)= {\tt {PGF}}^{(2)}_{\mathcal{B}}(R)$, where we denote by ${\tt {PGF}}^{(2)}_{\mathcal{B}}(R)$ the class of R-modules $M$ which are syzygies in an acyclic complex $\textbf{G}$ of modules in ${\tt PGF}_{\mathcal{B}}(R)$, such that $B\otimes_R\textbf{G}$ is acyclic for every $R$-module $B\in \mathcal{B}$.
	
(iii) An application of Theorem 5.7 may be found in \cite[Proposition 2.15]{St}.
\end{Remark}
\section*{Acknowledgements}The authors wish to thank Ioannis Emmanouil, Jan Saroch and the anonymous referee for useful comments.
%\section*{Declarations}
%\noindent\textbf{Conflict of interest.} The authors declare that they have no conflict of interest.

\bigskip
{\small {\sc Department of Mathematics,
             University of Athens,
             Athens 15784,
             Greece}}

{\em E-mail address:} {\tt kaperonn@math.uoa.gr}

\bigskip
{\small {\sc Department of Mathematics,
		University of Athens,
		Athens 15784,
		Greece}}
	
{\em E-mail address:} {\tt dstergiop@math.uoa.gr}

\end{document}